\documentclass[letterpaper, 10 pt, conference]{ieeeconf}  

\usepackage{graphics} 
\usepackage{times} 
\usepackage{amsmath} 
\usepackage{amssymb}  
\let\labelindent\relax
\usepackage{enumitem}
\usepackage{mathtools}
\usepackage[colorlinks=true,
            linkcolor=black]{hyperref}

\newtheorem{theorem}{\bf Theorem}

\newtheorem{remark}{\bf Remark}

\newtheorem{proposition}{\bf Proposition}
\newtheorem{lemma}{\bf Lemma}

\def\QED{~\rule[-1pt]{5pt}{5pt}\par\medskip}

\newcommand{\bx}{{\bf x}}
\newcommand{\by}{{\bf y}}
\newcommand{\bz}{{\bf z}}
\newcommand{\bu}{{\bf u}}
\newcommand{\bv}{{\bf v}}
\newcommand{\bw}{{\bf w}}

\newcommand{\be}{{\bf e}}
\newcommand{\bq}{{\bf q}}
\newcommand{\bn}{{\bf n}}

\newcommand{\boldxi}{\boldsymbol \xi}
\newcommand{\boldeta}{\boldsymbol \eta}

\newcommand{\boldtheta}{\boldsymbol \theta}

\newcommand{\PP}{\mathbb{P}}

\newcommand*{\LONGVERSION}{}

\title{\LARGE \bf
Rate of Prefix-free Codes in LQG Control Systems
}

\author{Takashi Tanaka, Karl Henrik Johansson, Tobias Oechtering, Henrik Sandberg, Mikael Skoglund\\
School of Electrical Engineering, KTH Royal Institute of Technology 
}

\begin{document}

\maketitle
\thispagestyle{empty}
\pagestyle{empty}

\begin{abstract}
In this paper, we consider a discrete time linear quadratic Gaussian (LQG) control problem in which state information of the plant is encoded in a variable-length binary codeword at every time step, and a control input is determined based on the codewords generated in the past.
We derive a lower bound of the rate achievable by the class of prefix-free codes attaining the required LQG control performance. This lower bound coincides with the infimum of a certain directed information expression, and is computable by semidefinite programming (SDP). 
Based on a technique by Silva et al., we also provide an upper bound of the best achievable rate by constructing a controller equipped with a uniform quantizer with subtractive dither and Shannon-Fano coding.
The gap between the obtained lower and upper bounds is less than $0.754r+1$ bits per time step regardless of the required LQG control performance, where $r$ is the rank of a signal-to-noise ratio matrix obtained by SDP, which is no greater than the dimension of the state.
\end{abstract}

\section{Introduction}

Motivated by control systems implemented by digital computers, we consider a discrete-time optimal control problem in which  state information of the plant is encoded using a variable-length binary sequence (codeword) at every time step, and control actions are determined based on the codewords generated in the past.
The performance of such a control system is characterized in a trade-off between control theoretic (e.g., the LQG control cost, $\gamma$) and communication theoretic (e.g., expected codeword length in average, $\mathsf{R}$) criteria. We say that a pair $(\gamma, \mathsf{R})$ is achievable if there exists a design attaining the control performance $\gamma$ and the rate $\mathsf{R}$. 
Understanding  the achievable $(\gamma, \mathsf{R})$ region is of great interest from both theoretical and practical perspectives. 

Unfortunately, explicit descriptions of the achievable regions are rarely available, even for relatively simple control problems. 
Consequently, our focus is to obtain tight inner and outer bounds of the region, or equivalently, upper and lower bounds of the trade-off function $\mathsf{R}(\gamma)$ that carves out the achievable region.
Recently, Silva et al. \cite{silva2011} showed that the rate of an arbitrary prefix-free code, if it is used in feedback control, is lower bounded by the \emph{directed information} from the output $y$ of the plant to the control input $u$.
They also showed that the conservativeness of this lower bound is strictly less than $\frac{1}{2}\log\frac{2\pi e}{12}+1\approx 1.254$ bits per time step, by constructing an entropy coded dithered quantizer (ECDQ) achieving this performance.
These observations suggest that directed information is a relevant quantity to study the best achievable rate by prefix-free codes in a control system.

Following this suggestion, we start our discussion with a characterization of the minimum directed information, denoted by $\mathsf{DI}(\gamma)$,\footnote{This quantity is related to the sequential rate-distortion function \cite{tatikonda2004}.} that needs to be ``processed" by any control law (without quantization and coding) in order to achieve the desired LQG control performance $\gamma$.
It turns out that finding $\mathsf{DI}(\gamma)$ is a convex optimization problem, and the previous discussion implies $\mathsf{DI}(\gamma) \leq \mathsf{R}(\gamma)$.
Then, invoking the idea of \cite{silva2011} and \cite{derpich2012}, we construct a control system involving an ECDQ that attains the LQG control performance $\gamma$.
Using standard properties of dithered quantizers \cite{zamir1992universal}, we then show that the rate of the designed controller is less than $\mathsf{DI}(\gamma)+\frac{r}{2}\log\frac{4\pi e}{12}+1$, where $r$ is some integer no greater than the state space dimension of the plant.
This establishes upper and lower bounds of $\mathsf{R}(\gamma)$ as
\[
\mathsf{DI}(\gamma) \leq \mathsf{R}(\gamma) < \mathsf{DI}(\gamma)+\frac{r}{2}\log\frac{4\pi e}{12}+1,
\]
which is the main result of this paper.

Our result is applicable to MIMO plants, while the result of \cite{silva2011} is restricted to SISO plants.
The restriction there is due to the difficulty of obtaining an analytical expression of $\mathsf{DI}(\gamma)$ in MIMO cases and a systematic method to design vector quantizers.\footnote{In \cite{silva2011}, it is suggested to reformulate a rate-constrained control problem with an SNR-constrained control problem. See, e.g., \cite{elia2004bode,braslavsky2007feedback} for a related discussion. 
A quantizer is then designed to match the optimal SNR.
However, the SNR of MIMO quantizers are matrix-valued in general, and no result is available to obtain an optimal matrix-valued SNR.}  In this regard, a key contribution of this paper is the use of semidefinite programming (SDP)  \cite{1411.7632}, both in the computation of $\mathsf{DI}(\gamma)$ and in the construction of an ECDQ.\footnote{This technique was facilitated by the recent advancements in the sequential rate-distortion theory \cite{1510.04214, charalambous2014nonanticipative}.}
Although we considered fully observable plants in this paper, the results can be generalized to partially observable plants using an
SDP-based solution to the sequential rate-distortion problem for partially observable sources \cite{srdpartially}.

Throughout the paper, we consider uniform quantizers  simply in the interest of mathematical ease of analysis to obtain an upper bound.
Optimal quantizer design in general requires much more involved procedures. For instance,  problems over memoryless noisy channels with finite input alphabets are considered in \cite{bao2011iterative}, where an iterative encoder/controller design procedure is proposed. 
General treatments of joint quantizer/controller design, discussions towards structural results of optimal policies, and a historical review of related problems are available in  \cite[Ch. 10,11]{yuksel2013stochastic}.

After the problem formulation in Section~\ref{secformulation}, we derive a lower bound of $\mathsf{R}(\gamma)$ is Section~\ref{seclb}. We propose a concrete quantizer/controller design in Section~\ref{secq}, whose performance is analyzed in Section~\ref{secanalysis} to derive the main result.
 
%

\section{Problem formulation}
\label{secformulation}
We assume that the plant in Figure~\ref{fig:fb} has a linear time-invariant state space model
\begin{equation}
\label{eqsystem}
\bx_{t+1}=A\bx_t+B\bu_t+\bw_t
\end{equation}
where matrices $A\in\mathbb{R}^{n\times n}$ and $B\in\mathbb{R}^{n\times m}$ are known, the initial state $\bx_1\sim\mathcal{N}(0,P_{1|0})$ has a known prior with $P_{1|0}\succ 0$, and the process noise $\bw_t\sim\mathcal{N}(0,W)$ is i.i.d. with known $W\succ 0$.
At every time step $t=1,2,\cdots$, the ``sensor+encoder" block observes the state $\bx_t$ and produces a single codeword $\bz_t$ from a predefined set $\mathcal{Z}_t$ of at most countable codewords. 
Upon receiving $\bz_t$, the ``decoder+controller" block produces a control input $\bu_t$. In what follows, the ``sensor+encoder" block is simply referred to as the \emph{encoder}, and likewise the ``decoder+controller" block as the \emph{decoder}. 
Both encoder and decoder are allowed to have infinite memories of the past.
We assume there is no delay due to encoding and decoding processes.

In this paper, we restrict ourselves to the class of prefix-free (instantaneous) binary codewords $\bz_t$. We allow the codebook $\mathcal{Z}_t$ to be time-varying and countably infinite set (hence $\bz_t$ can be an arbitrarily long binary sequence).
We design a variable-length code where the length of the codeword $\bz_t$ generated at time step $t$ is a random variable denoted by $l_t$.

\begin{remark}
Note that prefix-free may not always be a strict requirement for a code used in feedback control, although this assumption is taken for granted in the previous work \cite{silva2011}.
For instance, if both the encoder and decoder have access to a common clock signal, and know that only one codeword is generated at a time, a set of codewords
\[
\{\phi\text{ (zero-length codeword)}, 0, 1, 00, 01, 10, 11, 000, \cdots \}
\]
can be used to decode a message without any confusion, even though these codewords are not uniquely decodable. Nevertheless, there are several practical advantages of using prefix-free codes. For instance, prefix-free allows us to decode without referring to the common clock signal, which may simplify the implementation of the algorithm. Thus, the analysis in this paper is restricted to prefix-free binary codes.\end{remark}
\begin{figure}[t]
    \centering
    \includegraphics[width=0.9\columnwidth]{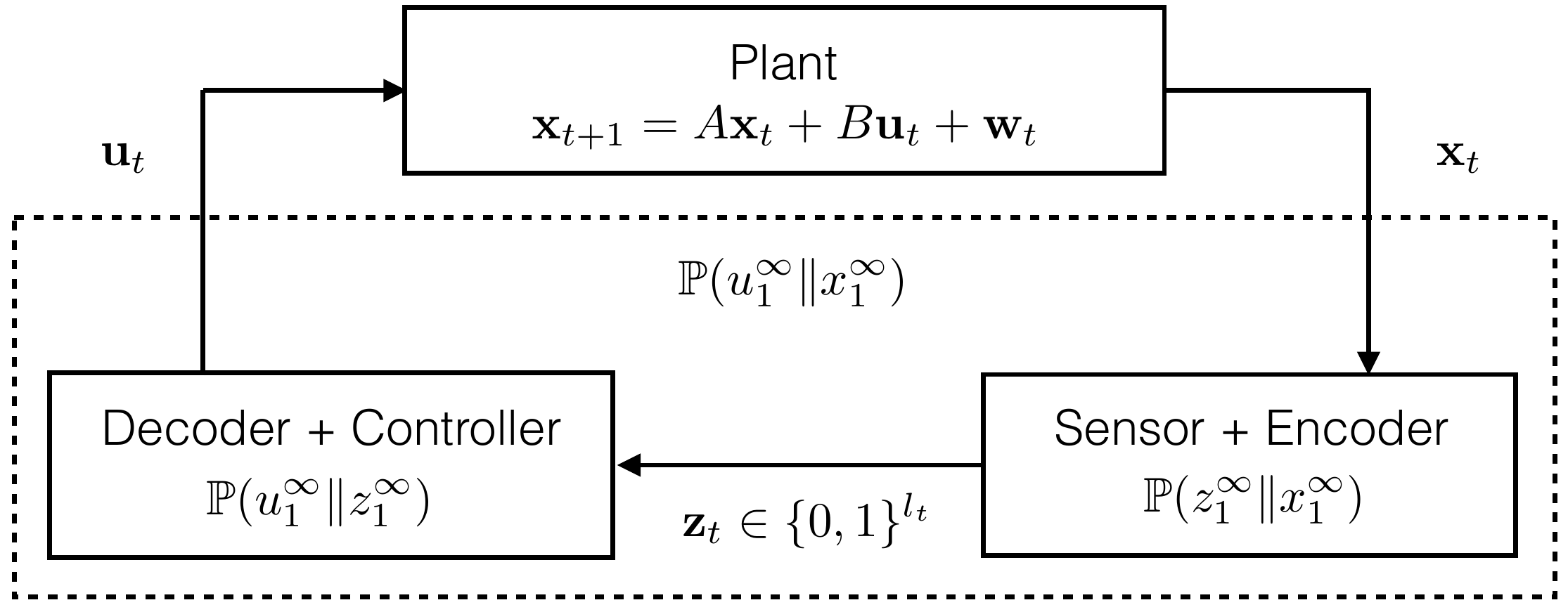} 
    \vspace{-1ex} 
    \caption{Feedback control using variable-length binary codewords.}
    \label{fig:fb}
    \vspace{-3ex}
\end{figure} 

We consider a joint design of a set of codewords $\mathcal{Z}_t$ for each $t$, the encoder's policy $\PP(z_1^\infty \| x_1^\infty)$, and the decoder's policy $\PP(u_1^\infty\|z_1^\infty)$. 
Here, we use Kramer's notation \cite{kramer2003capacity} for the sequence of causally conditioned Borel measurable stochastic kernels 
\begin{align}
\PP(z_1^\infty\| x_1^\infty) &= \{\PP(z_t|x^t,z^{t-1})\}_{t=1,2,\cdots} \label{eqkernel1}\\
\PP(u_1^\infty\| z_1^\infty) &= \{\PP(u_t|z^t,u^{t-1})\}_{t=1,2,\cdots}. \label{eqkernel2}
\end{align}
The purpose of our design is two-fold. First, we require that the overall control system achieves the LQG control cost $\leq \gamma$. Second, the expected codeword length on average is minimized. The optimization problem of our interest is 
\begin{align}
\mathsf{R}(\gamma)\triangleq \inf & \limsup_{T\rightarrow +\infty}\frac{1}{T}\sum\nolimits_{t=1}^T\mathbb{E} (l_t) \label{optRD} \\
\text{s.t. } & \limsup_{T\rightarrow +\infty}\frac{1}{T}\sum\nolimits_{t=1}^T\mathbb{E}\left(\|\bx_{t+1}\|_{Q}^2 + \|\bu_t\|_{R}^2 \right)\leq \gamma \nonumber
\end{align}
where $Q\succ 0$ and $R\succ 0$.
Expectations are evaluated with respect to the probability law induced by (\ref{eqsystem}), (\ref{eqkernel1}) and (\ref{eqkernel2}).
We assume that $(A,B)$ is stabilizable and $(A,Q)$ is detectable. 



\section{Lower Bound}
\label{seclb}

For every $\gamma>0$, define a function $\mathsf{DI}(\gamma)$ as the optimal value of the following convex optimization problem.
\begin{align}
\mathsf{DI}(\gamma) \triangleq \inf & \limsup_{T\rightarrow +\infty} \frac{1}{T}I(\bx^T\rightarrow \bu^T) \label{optdirectedinfo}\\
\text{s.t.} & \limsup_{T\rightarrow +\infty}\frac{1}{T}\sum\nolimits_{t=1}^T \mathbb{E}\left(\|\bx_{t+1}\|_{Q}^2 + \|\bu_t\|_{R}^2 \right)\leq \gamma. \nonumber
\end{align}
We use Massey's definition of the \emph{directed information} \cite{massey1990causality}:
\[
I(\bx^T\rightarrow \bu^T)\triangleq \sum\nolimits_{t=1}^T I(\bx^t;\bu_t|\bu^{t-1}).
\]
The infimum in (\ref{optdirectedinfo}) is taken over the sequence of causally conditioned Borel measurable stochastic kernels
\[
\PP(u_1^\infty\|x_1^\infty)\triangleq \{ \PP(u_t|x^t,u^{t-1})\}_{t=1,2,\cdots}.
\]
Under the aforementioned stabilizability/detectability assumption,  the optimization problem (\ref{optdirectedinfo}) is always feasible and $\mathsf{DI}(\gamma)<+\infty$. However, there is no need to solve an infinite-dimensional optimization problem (\ref{optdirectedinfo}) to compute $\mathsf{DI}(\gamma)$. 
\begin{proposition}
\label{propsdp}
(\cite{1510.04214})
Let $S$ be the unique positive definite solution to the algebraic Riccati equation
\begin{equation*}
A^\top SA-S-A^\top SB(B^\top SB+R)^{-1}B^\top SA+Q=0,
\end{equation*}
and $K\!\triangleq\! -(B^\top SB\!+\!R)^{-1}B^\top SA$, $\Theta\!\triangleq \!K^\top (B^\top SB\!+\!R)K$. Then $\mathsf{DI}(\gamma)$ is computable by semidefinite programming:
\begin{align}
\mathsf{DI}(\gamma)=\min_{P, \Pi\succ 0} & \quad \frac{1}{2} \log\det \Pi^{-1} + \frac{1}{2} \log \det W \label{optsdr}\\
\text{s.t.} & \quad \text{Tr}(\Theta P) + \text{Tr}(W S) \leq \gamma, \nonumber \\
& \quad  P\preceq A P A^\top +W, \nonumber \\
&\hspace{1ex} \left[\!\! \begin{array}{cc}P-\Pi \!\!\! &\!\! PA^\top \nonumber \\
AP \!\!\!&\!\! A PA^\top +W \end{array}\!\!\right]\! \succeq\! 0. \nonumber
\end{align} 
Let $P(\gamma)$ be an optimal solution to (\ref{optsdr}), and define
$
\mathsf{SNR}(\gamma)\triangleq P(\gamma)^{-1}-(AP(\gamma)A^\top+W)^{-1}$, and $r\triangleq \text{rank}(\mathsf{SNR}(\gamma))$.
Let $C\in\mathbb{R}^{r\times n}$ be a matrix with orthonormal columns and $V\in\mathbb{S}^r_{++}$ be a diagonal matrix satisfying $C^\top V^{-1} C=\mathsf{SNR}(\gamma)$. Then, an optimal solution $\PP(u_1^\infty\| x_1^\infty)$ to (\ref{optdirectedinfo}) can be realized by (i) an additive white Gaussian noise channel $\by_t=C\bx_t+\bv_t$ where $\bv_t\sim\mathcal{N}(0,V)$ is i.i.d., (ii) a Kalman filter $\hat{\bx}_t=\mathbb{E}(\bx_t|\by^t, \bu^{t-1})$, and (iii) a certainty equivalence controller $\bu_t=K\hat{\bx}_t$.  An equivalent block diagram is shown in  Figure~\ref{fig:threestage}.
\end{proposition}

The next result, appearing in \cite[Theorem 4.1]{silva2011}, claims that $\mathsf{DI}(\gamma)$ provides a lower bound of $\mathsf{R}(\gamma)$.
\begin{theorem}
\label{theolower}
For every $\gamma > 0$, we have $\mathsf{DI}(\gamma) \leq \mathsf{R}(\gamma)$.
\end{theorem}
\begin{proof}The inequality is directly verified as follows.
\begin{subequations}
\begin{align}
& \;I(\bx^T\rightarrow \bu^T) \\
\leq & \;I(\bx^T\rightarrow \bz^T \| \bu^{T-1}) \label{ineqchain1}\\
=&\sum\nolimits_{t=1}^T I(\bx^t;\bz_t| \bz^{t-1},\bu^{t-1}) \label{ineqchain2} \\
=&\sum\nolimits_{t=1}^T (H(\bz_t|\bz^{t-1},\bu^{t-1})-H(\bz_t|\bx^t,\bz^{t-1},\bu^{t-1})) \\
\leq &\sum\nolimits_{t=1}^T H(\bz_t|\bz^{t-1},\bu^{t-1}) \\
\leq &\sum\nolimits_{t=1}^T H(\bz_t) \\
\leq &\sum\nolimits_{t=1}^T \mathbb{E}(l_t) \label{ineqchain3}
\end{align}
\end{subequations}
The first step (\ref{ineqchain1}) is due to the \emph{feedback data processing inequality} discussed in \cite{silva2011,1510.04214}. The definition of causally conditioned directed information \cite{kramer2003capacity} is used in step   (\ref{ineqchain2}). The final step (\ref{ineqchain3}) is due to the fact that any prefix-free code is a prefix-free code of itself, and its expected length is greater or equal to its entropy \cite[Theorem 5.3.1]{CoverThomas}.
\end{proof}
\begin{figure}[t]
    \centering
    \includegraphics[width=0.95\columnwidth]{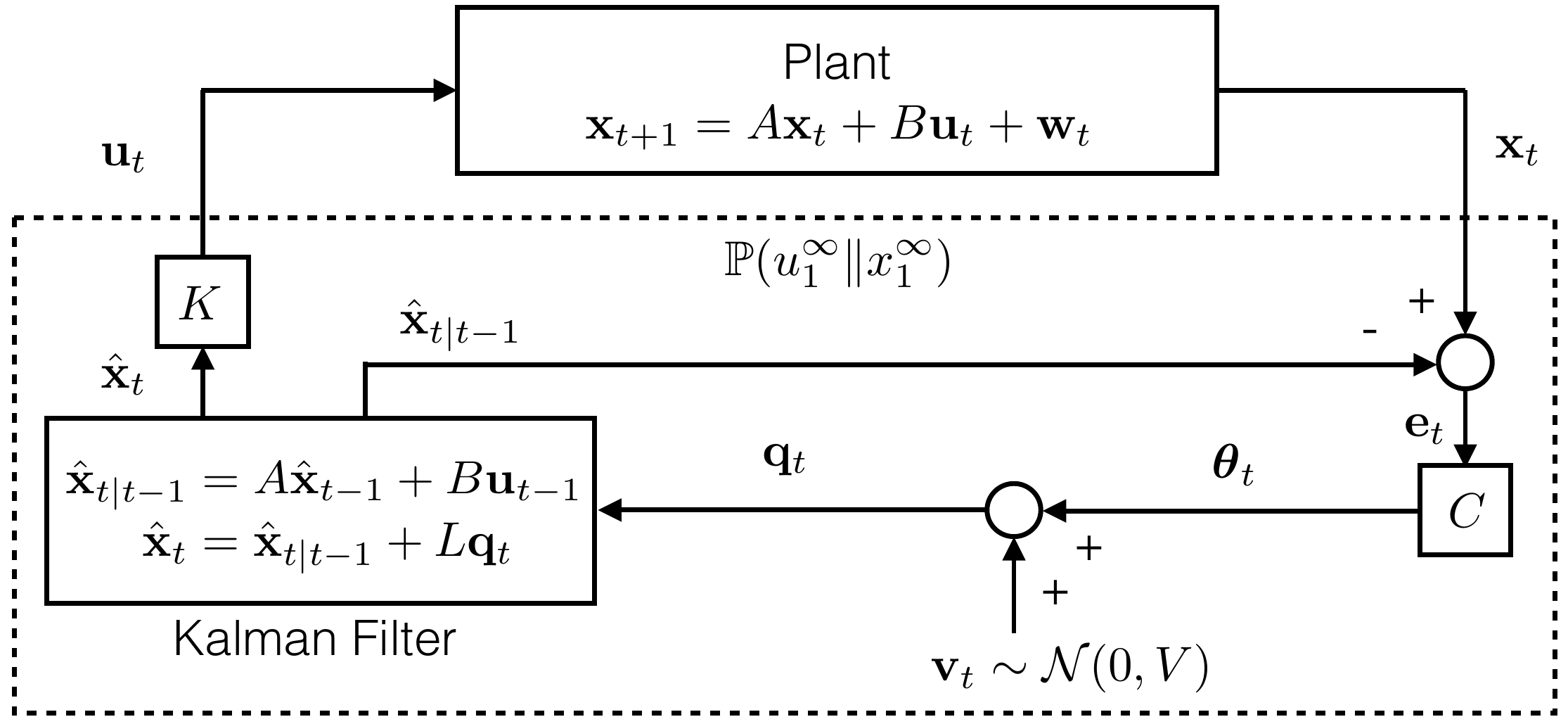}
    \vspace{-1ex} 
    \caption{A realization of an optimal solution $\PP(u_1^\infty\| x_1^\infty)$ to problem (\ref{optdirectedinfo}). Matrix $L$ is the optimal Kalman gain.}
    \label{fig:threestage}
    \vspace{-3ex}
\end{figure}
Note that there exists a nonsingular code whose expected codeword length is less that the entropy of the source \cite[Problem 5.31]{CoverThomas}.
Hence, Theorem \ref{theolower} does not hold in general if we allow a code that is not uniquely decodable. 

It can also be shown \cite{1510.04214} that, if the control system in Figure~\ref{fig:threestage} is designed according to the procedure in Proposition~\ref{propsdp}, and if $\by_t= C\bx_t+\bv_t$, then
\begin{equation}
\label{dixuxy}
I(\bx^T\rightarrow \bu^T)=\sum\nolimits_{t=1}^T I(\bx_t;\by_t|\by^{t-1}).
\end{equation}

\section{Uniform quantization with entropy coding}
\label{secq}
The rest of the paper is devoted to obtain an upper bound of the best achievable rate (\ref{optRD}).
To this end, we consider a concrete quantization/coding scheme and analyze its worst case rate. 
For the ease of analysis, we focus on a uniform scalar quantization with subtractive dither followed by entropy coding, following the ideas of \cite{silva2011,derpich2012}.
However, our coding scheme is designed based on the SDP-based solution to (\ref{optdirectedinfo}), which did not appear previously.

Let $Q_\Delta(\cdot)$ be a scalar quantizer defined by
\[
Q_\Delta(\bx)=i\Delta \;\; \text{ for } i\Delta -\tfrac{\Delta}{2}\leq \bx < i\Delta +\tfrac{\Delta}{2}.
\]
A scalar quantizer with subtractive dither $Q_\Delta^{\text{S.D.}}(\cdot)$ is defined by $Q_\Delta^{\text{S.D.}}(\bx)=Q_\Delta(\bx+\boldxi)-\boldxi$, where $\boldxi$ is a random variable uniformly distributed over $[-\Delta/2, \Delta/2]$.
\begin{remark}
$Q_\Delta^{\text{S.D.}}(\cdot)$ has some convenient mathematical properties (presented in Lemma~\ref{prop2} below) that will simplify the rate analysis. However, note that implementation of $Q_\Delta^{\text{S.D.}}(\cdot)$ requires a shared randomness $\boldxi$ both at the encoder's and the decoder's ends. 
In practice, two synchronized pseudorandom number generators can be used at the both ends.
\end{remark}

\subsection{Predictive quantizer}
The knowledge of the structure of an optimal solution $\PP(u_1^\infty\|x_1^\infty)$ to (\ref{optdirectedinfo}) can be used for an efficient quantizer/encoder design.
Intuitively, we design a coding scheme in such a way that the encoder and decoder policies jointly define a stochastic kernel that is similar to $\PP(u_1^\infty\|x_1^\infty)$.
This can be done, as specified below, by selecting the quantizer step size $\Delta$ so that the covariance of the resulting quantization error matches $V$. 
However, the encoder shall not quantize the observed state $\bx_t$ directly. 
In order to minimize the entropy of the quantizer input while keeping the contained information statistically equivalent, it is more advantageous to quantize the deviation of $\bx_t$ from the linear least mean square estimate $\hat{\bx}_{t|t-1}$ of $\bx_t$ given $(\by^{t-1}, \bu^{t-1})$. This technique is related to the
\emph{innovations approach} \cite{kailath1968innovations}, which is also used in \cite{borkar1997lqg}.\footnote{It is known that a predictive quantizer can be used without loss of performance in the joint control/quantizer design for LQG systems \cite{bao2011iterative,yuksel2014jointly}.}

In particular, consider a feedback control system illustrated in Figure~\ref{fig:block}.
Matrices $C$, $K$ and $L$ are chosen to be the same as in Figure~\ref{fig:threestage}.
Based on the output $\hat{\bx}_{t|t-1}$ of the Kalman filter block, consider a (scaled) estimation error
\[
\boldtheta_t=C(\bx_t-\hat{\bx}_{t|t-1}).
\]
Note that this is an $\mathbb{R}^r$-valued random process.
Let $V_i>0, i=1,\cdots, r$ be the $i$-th diagonal entry of $V$, and choose $\Delta_i>0, i=1,\cdots, r$ such that $\tfrac{\Delta_i^2}{12}=V_i$.
We apply the uniform quantizers with step sizes $\Delta_i$ with subtractive dither separately to each component of $\boldtheta_t$, i.e.,
\[
\bq_{t,i}=Q_{\Delta_i}^{\text{S.D.}}(\boldtheta_{t,i}), i=1,\cdots, r
\]
and define an $\mathbb{R}^r$-valued process $\bq_t=(\bq_{t,1},\cdots, \bq_{t,r})$.
More precisely, letting $\boldxi_t=(\boldxi_{t,1},\cdots, \boldxi_{t,r})$ be an $\mathbb{R}^r$-valued dither signal whose components are mutually independent and $\boldxi_{t,i}\sim\mathcal{U}[-\tfrac{\Delta_i}{2},\tfrac{\Delta_i}{2}]$, the quantizer output is given by
\[
\tilde{\bq}_{t,i}=Q_{\Delta_i}(\boldtheta_{t,i}+\boldxi_{t,i}), i=1,\cdots, r.
\]
Notice that $\tilde{\bq}_t=(\tilde{\bq}_{t,1},\cdots,\tilde{\bq}_{t,r})$ takes countably infinite possible values.
At every time step $t$, we apply an entropy coding scheme (described below) to $\tilde{\bq}_t$ to generate a codeword $\bz_t$.
The decoder then reproduces $\tilde{\bq}_t$ and $\bq_t$ by subtracting the dither signal.

Notice that the Kalman filter in Figure~\ref{fig:block} may not be the least mean square estimator anymore since signals are no longer Gaussian because of the dithered quantizers. 

\begin{figure}[t]
    \centering
    \includegraphics[width=\columnwidth]{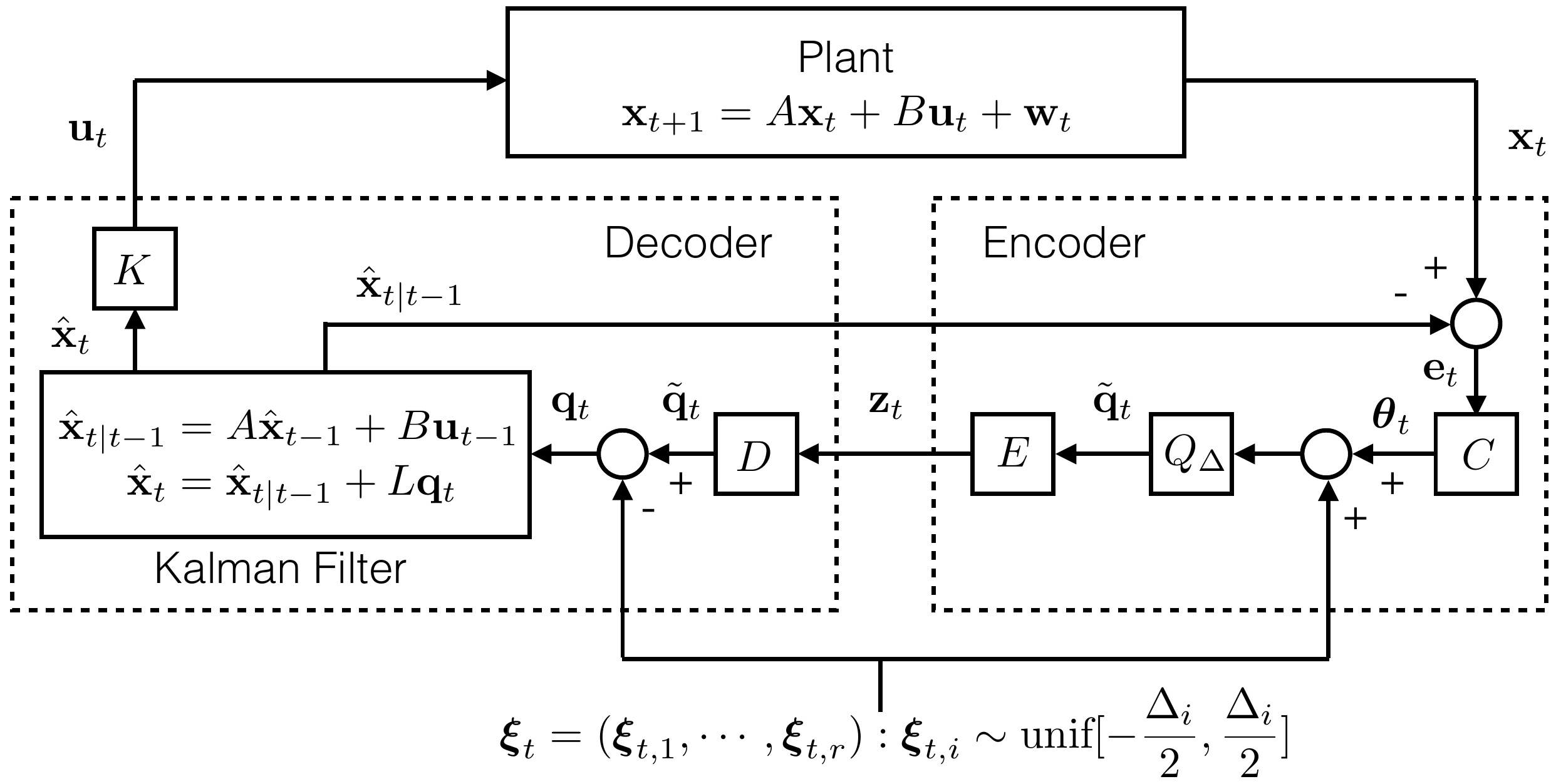} 
    \vspace{-4ex} 
    \caption{Proposed feedback control architecture. For simplicity, a feedback link from the decoder to encoder that transmits $\hat{\bx}_{t|t-1}$ is shown, but this link is unnecessary in practice, since the value of $\hat{\bx}_{t|t-1}$ can be completely estimated by the encoder alone. }
    \label{fig:block}
    \vspace{-3ex}
\end{figure}

\subsection{Entropy coding}
\label{secentropycoding}
At every time step $t$, we assume that a message $\tilde{\bq}_t$ is mapped to a codeword $\bz_t\in\{0,1\}^{l_t}$ designed by Shannon-Fano coding.
\begin{theorem}
\label{Theoshannoncode}(Shannon-Fano code, \cite[Problem 5.28]{CoverThomas})
Let $\bx$ be a random variable that takes countably infinite values $\{1,2,\cdots\}$ with probabilities $p_1, p_2, \cdots$. Assume $p_i>0$ and $p_i\geq p_{i+1}$ for all $i=1,2,\cdots$. Define $F_i=\sum_{k=1}^{i-1} p_k$, and let a codeword $W_i$ for $i$ to be the binary expression of $F_i$ rounded off to $l_i$ digits, where $l_i=\lceil \log\frac{1}{p_i} \rceil$. Then the constructed code is prefix-free and the average length $\mathbb{E}(l)$ satisfies $H(\bx)\leq \mathbb{E}(l) \leq H(\bx)+1$. 
\end{theorem}

Theorem~\ref{Theoshannoncode} implies that for each $t$ there exists a code whose expected codeword length satisfies $H(\tilde{\bq}_t|\boldxi_t) \leq \mathbb{E}(l_t) \leq H(\tilde{\bq}_t|\boldxi_t)+1$. Here, we only consider time-invariant codebooks whose codewords are adapted to the stationary conditional distribution of $\tilde{\bq}_t$ given $\boldxi_t$.
\begin{remark}
However, for a fixed code, obtaining the true stationary distribution of $\tilde{\bq}_t$ given $\boldxi_t$ may not be a trivial process. In this context, one may need an online algorithm to estimate the true conditional distribution, or a variation of the universal coding that adapts its codewords.
\end{remark}

\section{Analysis}
\label{secanalysis}
In this section, we analyze the LQG control performance and the worst case rate of the encoding/control law designed in the previous section.
We will make use of the following technical lemma, which is a straightforward extension of the main results of \cite{zamir1992universal}. 
\ifdefined\LONGVERSION Proofs are provided in Appendix. \fi
\ifdefined\SHORTVERSION Proofs are provided in \cite{extendedversion}. \fi
\begin{lemma}
\label{prop2}
Let $\bx$ be an $\mathbb{R}^r$-valued random variable such that $h(\bx)$ is finite, and $\Delta=(\Delta_1,\cdots,\Delta_r)$ be a tuple of positive quantizer step sizes.
Let $\boldxi$ be the $\mathbb{R}^r$-valued dither independent of $\bx$ such that $\boldxi_i\sim\mathcal{U}[-\tfrac{\Delta_i}{2},\tfrac{\Delta_i}{2}]$, and suppose entries of $\boldxi$ are mutually independent.
Let $\bq$ and $\tilde{\bq}$ be $\mathbb{R}^r$-valued random variables whose $i$-th entries are defined by
\begin{align*}
\bq_i&=Q_{\Delta_i}(\bx_i+\boldxi_i)-\boldxi_i \left(=Q_{\Delta_i}^{\text{S.D.}}(\bx_i)\right) \\
\tilde{\bq}_i&=Q_{\Delta_i}(\bx_i+\boldxi_i).
\end{align*}
Then, the following statements hold.
\begin{itemize}[leftmargin=*]
\item[(a)] The $i$-th entry of the quantization error $\boldeta_i=\bq_i-\bx_i$ is independent of $\bx$ and is uniformly distributed on $[-\tfrac{\Delta_i}{2},\tfrac{\Delta_i}{2}]$. Moreover, entries of $\boldeta$ are mutually independent.
\item[(b)] Let $\bn$ be an $\mathbb{R}^r$-valued random variable. Suppose that $\bn$ is independent of $\bx$, entries of $\bn$ are mutually independent, and $\bn_i\sim\mathcal{U}[-\tfrac{\Delta_i}{2},\tfrac{\Delta_i}{2}]$ for $i=1,\cdots, r$. 
If $\by=\bx+\bn$, then
\[
H(\tilde{\bq}|\boldxi)=H(\bq|\boldxi)=h(\by)-\sum\nolimits_{i=1}^r \log \Delta_i=I(\bx;\by).
\]
\item[(c)] Let 
$
\textsf{RDF}_\bx(D)=\inf\nolimits_{\PP(u|x):\mathbb{E}\|\bx-\bu\|^2\leq D} I(\bx;\bu)
$
be the rate-distortion function of the source $\bx$, and
$
\textsf{C}_{\Delta}(D)=\sup\nolimits_{\PP(x):\mathbb{E}\|\bx\|^2\leq D} I(\bx;\by)
$
be the capacity of a channel shown in Figure~\ref{fig:channel}, where $D=\sum_{i=1}^r \tfrac{\Delta_i^2}{12}$. 
Then, we have 
\[
H(\tilde{\bq}|\boldxi) - \textsf{RDF}_\bx(D) \leq \textsf{C}_{\Delta}(D).
\]
\item[(d)] For every positive $\Delta_1, \cdots, \Delta_r$ and $D$ such that $D=\sum_{i=1}^r \tfrac{\Delta_i^2}{12}$, the capacity $\textsf{C}_\Delta(D)$ defined above satisfies
\[\mathsf{C}_\Delta(D)<\frac{r}{2}\log \frac{4\pi e}{12}.\]
\end{itemize}
\end{lemma}
\begin{figure}[t]
    \centering
    \includegraphics[width=0.5\columnwidth]{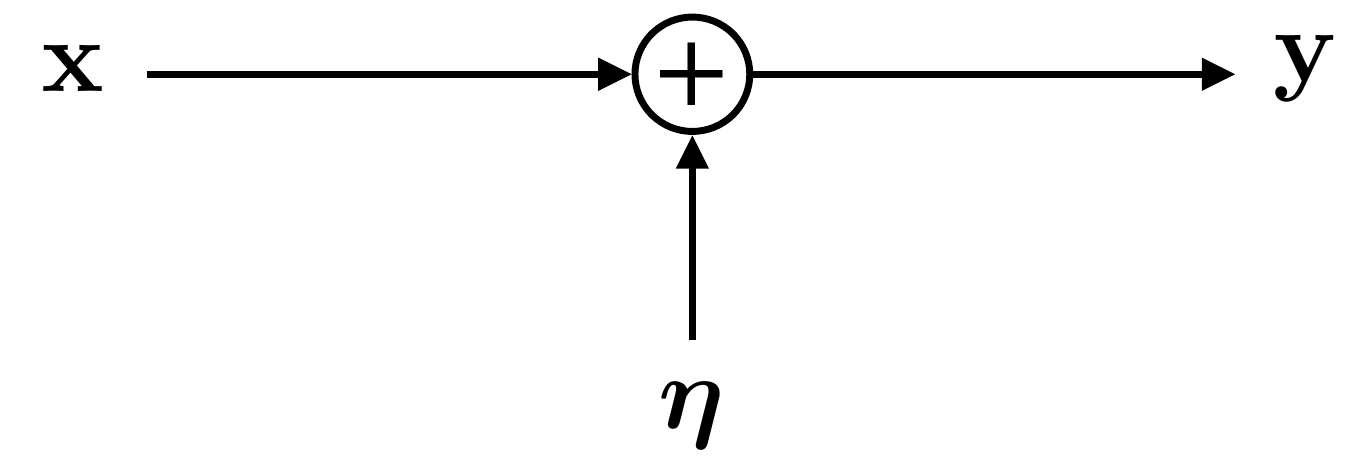} 
    \vspace{-1ex} 
    \caption{Auxiliary $r$-dimensional additive noise channel with input $\bx$ and output $\by=\bx+\boldeta$. Additive noise $\boldeta$ is independent of $\bx$, and $\boldeta_i\sim\mathcal{U}[-\tfrac{\Delta_i}{2},\tfrac{\Delta_i}{2}]$ for $i=1,\cdots,r$. Channel is subject to an input power constraint $\mathbb{E}\|\bx\|^2 \leq D =\sum_{i=1}^r \Delta_i^2/12$.}
    \label{fig:channel}
    \vspace{-3ex}
\end{figure}

\subsection{Control performance}
An implication of Lemma~\ref{prop2}(a) is that the uniform quantizer with subtractive dither can be equivalently modeled as an additive uniform noise channel.
Thus, Figure~\ref{fig:block} can be written as Figure~\ref{fig:equivsystem}, and in both figures, $\bx_t$ and $\bu_t$ have the same stationary joint distributions. 
Also, notice that the only difference between Figure~\ref{fig:threestage} and Figure~\ref{fig:equivsystem} is the noise statistics of $\bv_t$ and $\boldeta_t$.
In order to distinguish these two cases, denote by $(\bx_t^{\text{G}},\bu_t^{\text{G}})$ the jointly Gaussian random variables having the stationary joint distribution of $(\bx_t,\bu_t)$ in Figure~\ref{fig:threestage}, and by $(\bx_t^{\text{NG}},\bu_t^{\text{NG}})$ the non-Gaussian random variables having the stationary joint distribution of $(\bx_t,\bu_t)$ in Figure~\ref{fig:block} and \ref{fig:equivsystem}.
\begin{lemma}
\label{lem:samecov}
The joint distributions of $(\bx_t^{\text{G}},\bu_t^{\text{G}})$ and
$(\bx_t^{\text{NG}},\bu_t^{\text{NG}})$ have the same mean and covariance.
\end{lemma}
\begin{proof}
Notice that we have chosen the dither step sizes in such as way that $\boldeta_t$ has a covariance matrix $V$. Since all operations in Figures~\ref{fig:threestage} and \ref{fig:equivsystem} are linear, it can be shown by induction (in $t$) that $\bx_t$ and $\bu_t$ in Figures~\ref{fig:threestage} and \ref{fig:equivsystem} have the same first and second moments for all $t=1,2,\cdots$.
\end{proof}

Since the LQG control performance depends only on the second moments of the stationary distribution, it follows from Lemma~\ref{lem:samecov} that the control system in Figure~\ref{fig:block} attains the same control performance as Figure~\ref{fig:threestage} which is $\gamma$.

\begin{figure}[t]
    \centering
    \includegraphics[width=0.95\columnwidth]{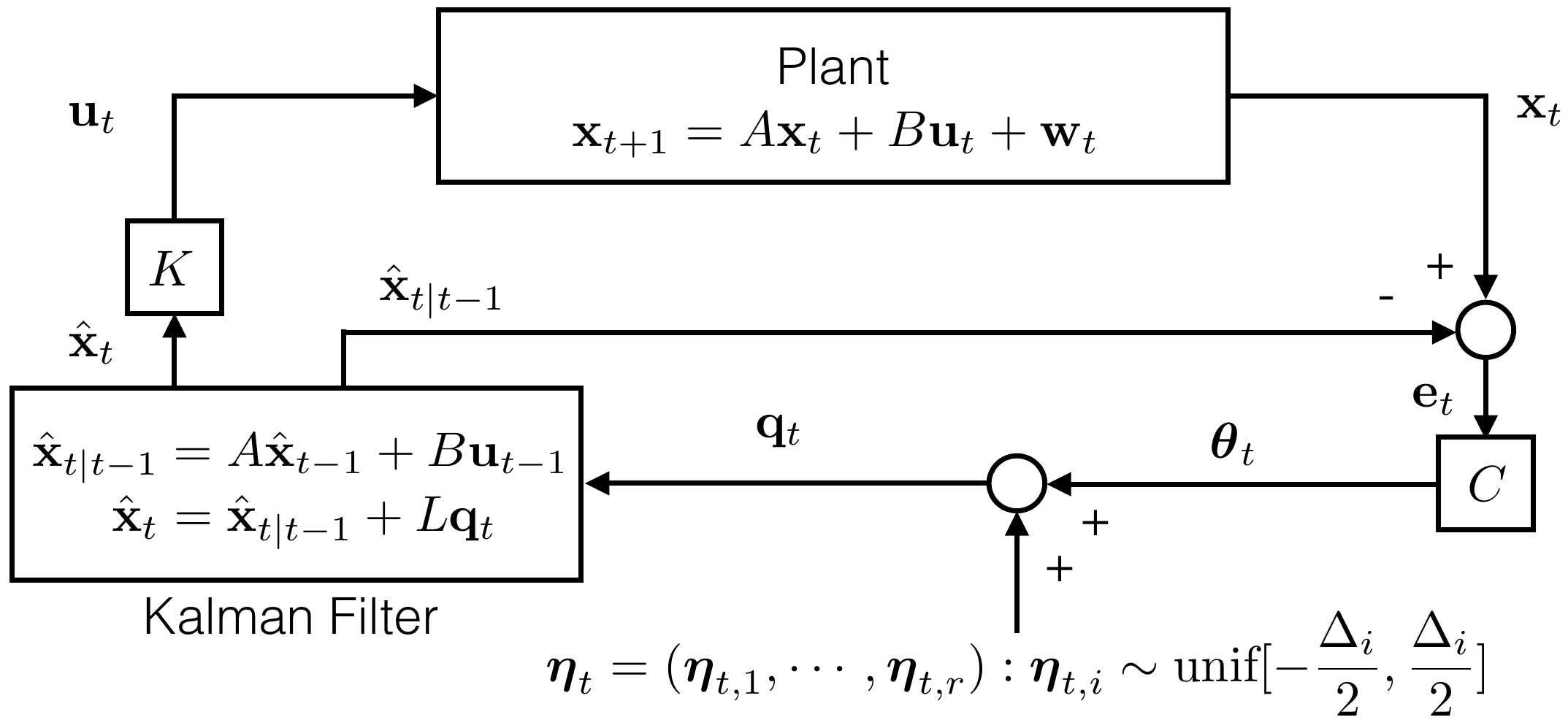} 
    \vspace{-2ex} 
    \caption{Equivalent system.}
    \label{fig:equivsystem}
    \vspace{-3ex}
\end{figure}

\subsection{Rate analysis}
We are now ready to prove the main result of this paper.
\begin{theorem}
\label{theoupper}
For every $\gamma>0$, we have 
\[
\mathsf{R}(\gamma) < \mathsf{DI}(\gamma)+\frac{r}{2}\log\frac{4\pi e}{12}+1\approx \mathsf{DI}(\gamma)+0.754r+1
\]
where $r=\text{rank}(\mathsf{SNR}(\gamma))$.
\end{theorem}
\begin{proof}
Set $D=\sum_{i=1}^r \frac{\Delta_i^2}{12}=\text{Tr} (V)$. 
If $\boldtheta_t^{\text{G}}$ is a random variable with the stationary distribution of $\boldtheta_t$ in Figure~\ref{fig:threestage} and $\boldtheta_t^{\text{NG}}$ is a random variable with the stationary distribution of $\boldtheta_t$ in Figure~\ref{fig:equivsystem}, 
we have
\begin{subequations}
\label{finproof}
\begin{align}
\mathbb{E}(l_t) &\leq H(\tilde{\bq}_t|\boldxi_t)+1 \label{finproof1} \\
&< \mathsf{RDF}_{\boldtheta_t^{\text{NG}}}(D)+\frac{r}{2}\log\frac{4\pi e}{12}+1 \label{finproof3} \\
& \leq \mathsf{RDF}_{\boldtheta_t^{\text{G}}}(D)+\frac{r}{2}\log\frac{4\pi e}{12}+1. \label{finproof4}
\end{align}
\end{subequations}
Step (\ref{finproof1}) follows from the discussion in Section~\ref{secentropycoding}, while Lemma~\ref{prop2} is used in (\ref{finproof3}). 
The last inequality (\ref{finproof4}) follows from the fact that $\boldtheta_t^{\text{G}}$ is Gaussian with the same mean and covariance as $\boldtheta_t^{\text{NG}}$, and that among all distributions sharing the same mean and covariance, the Gaussian one has the largest rate-distortion function \cite[Problem 10.8]{CoverThomas}.

Next, notice that for $T=1,2,\cdots$, the following inequality holds for random variables in Figure~\ref{fig:threestage}.
\begin{subequations}
\label{diinov}
\begin{align}
&I(\bx^T\rightarrow \bu^T)\\
&=\sum\nolimits_{t=1}^T I(\bx_t; \by_t|\by^{t-1}) \label{diinov1}\\
&=\sum\nolimits_{t=1}^T I(\bx_t\!-\!\hat{\bx}_{t|t-1}; C(\bx_t\!-\!\hat{\bx}_{t|t-1})\!+\!\bv_t|\by^{t-1}) \label{diinov2}\\
&=\sum\nolimits_{t=1}^T I(\be_t;C\be_t+\bv_t|\by^{t-1}) \label{diinov3}\\
&=\sum\nolimits_{t=1}^T I(\be_t;C\be_t+\bv_t) \label{diinov4}\\
&\geq \sum\nolimits_{t=1}^T I(\boldtheta_t;\boldtheta_t+\bv_t). \label{diinov5}
\end{align}
\end{subequations} 
The identity (\ref{dixuxy}) is used in step (\ref{diinov1}). Equation (\ref{diinov2}) holds since $\hat{\bx}_{t|t-1}$ is measurable with respect to the $\sigma$-algebra generated by $\by^{t-1}$. 
Due to the property of the least mean square error estimation, $\be_t$ is independent of $\by^{t-1}$. Since $\bv_t$ is also independent of $\by^{t-1}$, (\ref{diinov4}) holds. 
Finally, (\ref{diinov5}) is a consequence of the data processing inequality since $(\boldtheta_t+\bv_t)$ -- $\be_t$ -- $\boldtheta_t$ form a Markov chain.
Finally, for $t=1,2,\cdots$,
\begin{align}
I(\boldtheta_t;\boldtheta_t+\bv_t)
&\geq 
\begin{cases}
\min_{\PP(u_t|\theta_t)}  & I(\boldtheta_t;\bu_t) \\
\text{s.t. } & \mathbb{E}\|\boldtheta_t-\bu_t\|^2 \leq D 
\end{cases} \nonumber \\
&= \mathsf{RDF}_{\boldtheta_t^{\text{G}}}(D) \label{ineqIRDF}
\end{align}
since $\bu_t=\boldtheta_i+\bv_t$ satisfies the distortion constraint.

Combining (\ref{finproof}), (\ref{diinov}) and (\ref{ineqIRDF}), for $T=1,2,\cdots$, we obtain
$$
\frac{1}{T} \!\sum_{t=1}^T \mathbb{E}(l_t) < \frac{1}{T} I(\bx^T\!\rightarrow \!\bu^T)+\frac{r}{2}\log \!\frac{4\pi e}{12}+1.
$$
\end{proof}

For SISO plants, the gap obtained above is read as $\frac{1}{2}\log\frac{4\pi e}{12}+1\approx 1.754$. This value should not be confused with the value $1.254$ obtained in \cite{silva2011}, since considered lower bounds are not equivalent. In \cite{silva2011}, the lower bound is written in terms of the directed information from $\bx_t^{\text{NG}}$ to $\bu_t^{\text{NG}}$, while in Theorem~\ref{theoupper}, it is written in terms of the directed information from $\bx_t^{\text{G}}$ to $\bu_t^{\text{G}}$. 

If Shannon-Fano coding is applied to entries of $\tilde{\bq}_t$ separately, we obtain a larger gap of $\frac{r}{2}\log \frac{4\pi e}{12}+r$ bits per time step. If an appropriate lattice (vector) quantizer is used, we obtain a tighter gap of $\frac{r}{2}\log 4\pi e G_r+1$, where $G_r$ is the normalized second moment of the lattice \cite{zamir1992universal}.

\ifdefined\LONGVERSION
\section*{APPENDIX}
\subsection*{Proof of Lemma~\ref{prop2}}

(a) For $i=1,\cdots,r$, let $\bx_i^*=\bx_i-\Delta_i \lfloor \frac{\bx_i}{\Delta_i}+\frac{1}{2}\rfloor$ be the nearest value to the origin selected from 
\[
\cdots, \bx_i-2\Delta_i,  \bx_i-\Delta_i, \bx_i, \bx_i+\Delta_i, \bx_i+2\Delta_i,  \cdots.
\]
Then,
\begin{align*}
\boldeta_i&=Q_{\Delta_i}(\bx_i+\boldxi_i)-\boldxi_i-\bx_i \\
&=\begin{cases}
\bx_i^*-\boldxi_i-\Delta_i & \text{ for } -\frac{\Delta_i}{2} \leq \boldxi_i \leq -\bx_i^*-\frac{\Delta_i}{2} \\
\bx_i^*-\boldxi_i & \text{ for } -\bx_i^*-\frac{\Delta_i}{2}  < \boldxi_i \leq -\bx_i^*+\frac{\Delta_i}{2} \\
\bx_i^*-\boldxi_i+\Delta_i & \text{ for } -\bx_i^*+\frac{\Delta_i}{2} < \boldxi_i \leq \frac{\Delta_i}{2}.
\end{cases}
\end{align*}
For any realization of $\bx_i$, this function maps a density $\mathcal{U}[-\frac{\Delta_i}{2}, \frac{\Delta_i}{2}]$ on the $\boldxi_i$-axis to a density $\mathcal{U}[-\frac{\Delta_i}{2}, \frac{\Delta_i}{2}]$ on the $\boldeta_i$-axis.
Thus, $\boldeta_i\sim \mathcal{U}[-\frac{\Delta_i}{2}, \frac{\Delta_i}{2}]$ and this is independent of $\bx_i$. Moreover, if $i\neq j$, then $\boldeta_i$ and $\boldeta_j$ are independent, since $\boldxi_i$ and $\boldxi_j$ are independent.

(b) It is straightforward to prove the first and the third equalities. Thus we prove the second equality. Denote by
\[
\tilde{q}_i^{k_i}=k_i\Delta_i, \; k_i \in \mathbb{Z}
\]
the $k_i$-th possible value that $\tilde{\bq}_i$ can take.
Given a realization $\xi$ of the dither random variable $\boldxi$, $\bq$ takes discrete values of the form
\[
q^k=\tilde{q}^k-\xi, \; k\in\mathbb{Z}^r
\]
where $q^k=(q_1^{k_1},\cdots,q_r^{k_r})$ and $\tilde{q}^k=(\tilde{q}_1^{k_1},\cdots,\tilde{q}_r^{k_r})$. To compute $H(\bq |\boldxi)$, notice that the p.m.f. of $\bq$ given $\boldxi$ is
\begin{align}
P(q^k|\xi)&=\text{Prob}\left\{ \bigwedge_{i=1}^r \left(\tilde{q}_i^{k_i}-\tfrac{\Delta_i}{2}\leq \bx_i+\xi_i \leq \tilde{q}_i^{k_i}+\tfrac{\Delta_i}{2} \right) \right\} \nonumber \\
&=\text{Prob}\left\{ \bigwedge_{i=1}^r \left(q_i^{k_i}-\tfrac{\Delta_i}{2}\leq \bx_i \leq q_i^{k_i}+\tfrac{\Delta_i}{2} \right) \right\} \nonumber \\
&= \int_{\bigtimes_{i=1}^r[q_i^{k_i}-\frac{\Delta_i}{2},q_i^{k_i}+\frac{\Delta_i}{2}]} f_\bx(x)dx \label{qpmf}
\end{align}
which is a p.d.f. $f_\bx(\cdot)$ integrated over an $r$-dimensional hypercube.
The p.d.f. of $\by=\bx+\bn$ is given by
\[
f_\by(y)=\int_{\mathbb{R}^r} f_\bx(x)f_\bn(y-x)dx.
\]
Notice that $f_\bn(y-x)=\prod_{i=1}^r\frac{1}{\Delta_i}$ if $y-x$ is in the hypercube $\times_{i=1}^r[-\frac{\Delta_i}{2},\frac{\Delta_i}{2}]$, or equivalently if $x$ is in $\times_{i=1}^r[y_i-\frac{\Delta_i}{2},y_i+\frac{\Delta_i}{2}]$, and is zero otherwise. Hence
\begin{equation}
f_\by(y)=\left(\prod\nolimits_{i=1}^r \frac{1}{\Delta_i}\right)
\int_{\bigtimes_{i=1}^r[y_i-\frac{\Delta_i}{2},y_i+\frac{\Delta_i}{2}]} f_\bx(x) dx. \label{ypdf}
\end{equation}
Comparing (\ref{qpmf}) and (\ref{ypdf}), we can write
\[
P(q^k|\xi)=(\prod\nolimits_{i=1}^r \Delta_i)f_\by(q^k)=(\prod\nolimits_{i=1}^r \Delta_i)f_\by(\tilde{q}^k-\xi). 
\]
Thus $H(\bq |\boldxi)$ is given by
\begin{align*}
&H(\bq |\boldxi)\\
&=\mathbb{E}_{\boldxi} \left[ -\sum\nolimits_{k\in\mathbb{Z}^k} P(q^k|\xi)\log P(q^k|\xi) \right]\\
&=\mathbb{E}_{\boldxi} \left[ -\sum\nolimits_{k\in\mathbb{Z}^k} \left(\prod\nolimits_{i=1}^r \Delta_i\right)f_\by(\tilde{q}^k-\xi)\log f_\by(\tilde{q}^k-\xi) \right] \\
&\hspace{3ex}-\sum\nolimits_{i=1}^r \log\Delta_i \\
&=\int_{\bigtimes_{i=1}^r[-\frac{\Delta_i}{2},\frac{\Delta_i}{2}]}\left[-\sum\nolimits_{k\in\mathbb{Z}^k} f_\by(\tilde{q}^k-\xi)\log f_
\by(\tilde{q}^k-\xi)\right]d\xi \\
&\hspace{3ex}-\sum\nolimits_{i=1}^r \log\Delta_i \\
&=-\int_{\mathbb{R}^r} f_\by(y)\log f_\by(y)dy-\sum\nolimits_{i=1}^r \log\Delta_i \\
&=h(\by)-\sum\nolimits_{i=1}^r \log\Delta_i.
\end{align*}
 
(c) Suppose $\PP(u|x)$ attains the rate-distortion function and let $\bu$ be defined by $\PP(x)$ and $\PP(u|x)$. (If the infimum is not attained, alternatively consider a sequence of random variables $\bu^k$ such that $\PP(u^k|x)$ asymptotically attains the infimum.)
Without loss of generality, assume $\bu$ is independent of $\boldeta$. Then
\begin{subequations}
\begin{align}
H(\tilde{\bq}|\boldxi)-I(\bx;\bu)&=I(\bx;\by)-I(\bx;\bu) \label{lemcproof1}\\
&=I(\bx;\by|\bu)-I(\bx;\bu|\by) \label{lemcproof2}\\
&\leq I(\bx;\by|\bu) \\
&=h(\by|\bu)-h(\boldeta) \\
&=h(\by-\bu|\bu)-h(\boldeta) \\
&\leq h(\by-\bu)-h(\boldeta) \\
&=I(\bx-\bu;\by-\bu) \\
&\leq \mathsf{C}_\Delta (D) \label{lemcproof3}
\end{align}
\end{subequations}
The result of part (b) is used in (\ref{lemcproof1}), and (\ref{lemcproof2}) holds since
\begin{align*}
I(\bx;\by)\!-\!I(\bx;\bu)&=h(\bx)-h(\bx|\by)-h(\bx)+h(\bx|\by) \\
&=h(\bx|\bu)-h(\bx|\by) \\
&=h(\bx|\bu)\!-\!h(\bx|\by,\!\bu)\!-\!h(\bx|\by)\!+\!h(\bx|\by,\!\bu) \\
&=I(\bx;\by|\bu)-I(\bx;\bu |\by).
\end{align*}
The final inequality (\ref{lemcproof3}) holds by definition of $\mathsf{C}_\Delta (D)$, since the random variable $\bx-\bu$ satisfies the power constraint by construction of $\bu$.

(d) The claim is directly shown by the following chain of inequalities.
\begin{subequations}
\begin{align}
&\mathsf{C}_\Delta (D)=\sup_{\PP(x):\mathbb{E}\|\bx\|^2\leq D} I(\bx;\by) \\
&=\sup_{\PP(x):\mathbb{E}\|\bx\|^2\leq D}  h(\by)-\sum\nolimits_{i=1}^r \log \Delta_i \\
&=\sup_{\PP(x):\mathbb{E}\|\bx\|^2\leq D} \sum\nolimits_{i=1}^r \left( h(\bx_i+\boldeta_i)\!-\!\log \Delta_i \right)\\
&<\!\!\!\!\!\!\max_{{\tiny\begin{array}{c}p_i>0 \\ \sum_{i=1}^rp_i=D \end{array}}} \!\! \sum_{i=1}^r \left(\frac{1}{2}\log 2\pi e \left(\frac{\Delta_i^2}{12}+p_i \right)-\log \Delta_i \right) \label{lemdproof1}\\
&=\frac{r}{2}\log\frac{2\pi e}{12}+\!\!\!\!\max_{{\tiny\begin{array}{c}p_i>0 \\ \sum_{i=1}^rp_i=D \end{array}}} \!\! \sum_{i=1}^r\frac{1}{2}\log\frac{(\Delta_i^2/12)+p_i}{(\Delta_i^2/12)} \\
&=\frac{r}{2}\log\frac{2\pi e}{12}+\frac{r}{2}\log 2 \label{lemdproof2}\\
&= \frac{r}{2}\log\frac{4\pi e}{12}
\end{align}
\end{subequations}
In step (\ref{lemdproof1}), we assumed that the power $p_i$ is allocated to $\bx_i$. Since the covariance of $\boldeta_i$ is $\tfrac{\Delta_i^2}{12}$, the covariance of $\bx_i+\boldeta_i$ is $\tfrac{\Delta_i^2}{12}+p_i$.
The entropy of $\bx_i+\boldeta_i$ is upper bounded by the entropy of the Gaussian random variable with the same covariance. However, since $\bx_i+\boldeta_i$ cannot be Gaussian, the upper bound (\ref{lemdproof1}) is strict.
Finally, (\ref{lemdproof2}) is a simple application of the log-sum inequality.

\begin{figure}[t]
    \centering
    \includegraphics[width=0.45\columnwidth]{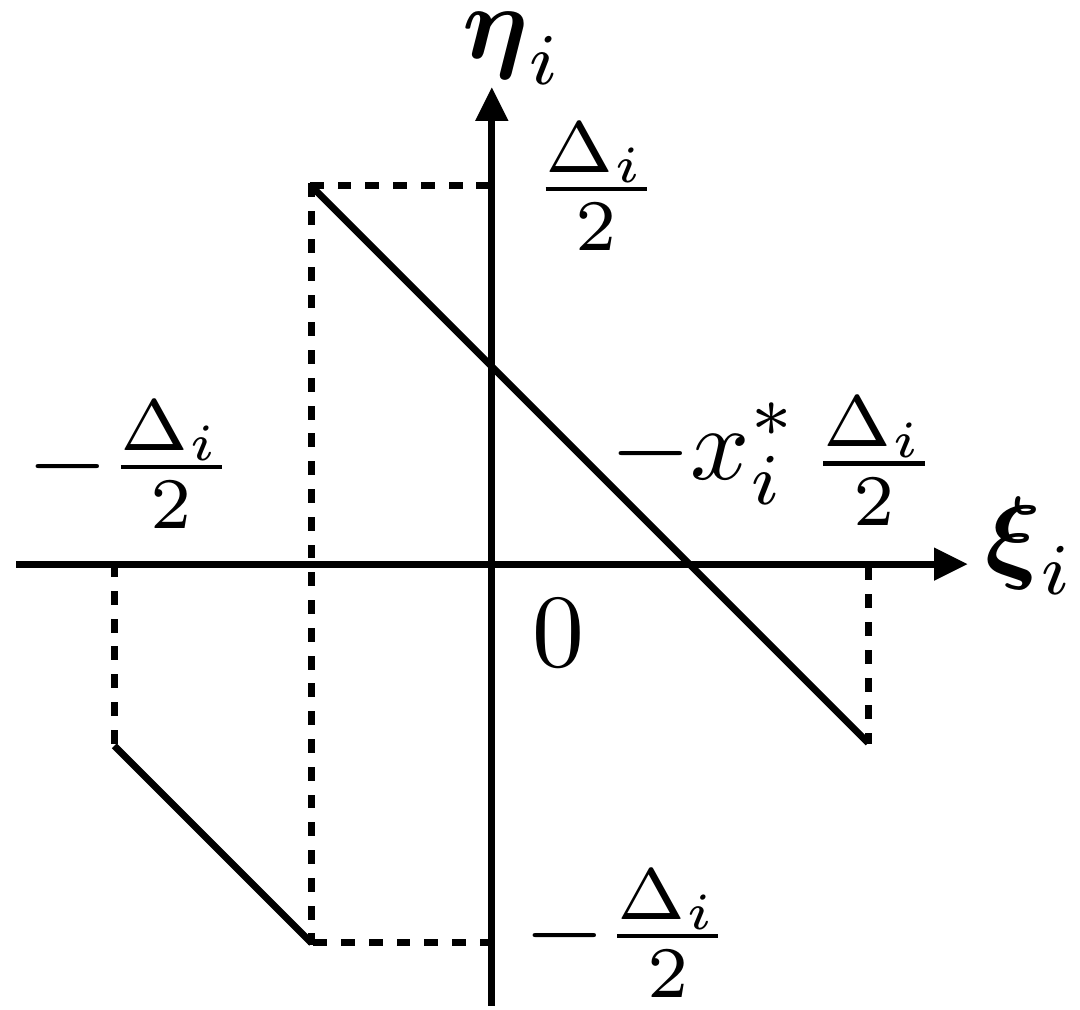} 
    \caption{Illustration of a map $\boldeta_i=Q_{\Delta_i}(x_i+\boldxi_i)-\boldxi_i-x_i$.}
    \label{fig:xieta}
    \vspace{-3ex}
\end{figure}

\fi


\bibliographystyle{IEEEtran}
\bibliography{ref}
\end{document}